\documentclass[12pt]{amsart}

\usepackage{graphicx}
\usepackage{amssymb}
\usepackage{amsthm}
\usepackage{amsmath}
\usepackage{stmaryrd}
\usepackage{cite}
\usepackage[mathscr]{eucal}
\usepackage{hyperref}
\usepackage[margin=1.0in]{geometry}
\usepackage{comment}

\usepackage{url}
\usepackage{hyperref}

\title[Tachibana-type Theorems]{Tachibana-type theorems and special Holonomy}
\author[Peter Petersen and Matthias Wink]{Peter Petersen and Matthias Wink
}
\address{Department of Mathematics, UCLA, 520 Portola Plaza, Los Angeles, CA, 90095}
\email{petersen@math.ucla.edu}

\address{Mathematisches Institut, Universit\"at M\"unster, Einsteinstra{\ss}e 62, 48149 M\"unster}
\email{mwink@uni-muenster.de}

\subjclass[2010]{32Q10, 32Q15, 32Q20, 53C20, 53C26\\ 
\text{} \hspace{2mm} MW funded by the Deutsche Forschungsgemeinschaft (DFG, German Research Foundation) under Germany's Excellence Strategy EXC 2044–390685587, Mathematics M\"unster: Dynamics–Geometry–Structure.}

\begin{document}
\newcommand{\Ext}{\bigwedge\nolimits}
\newcommand{\Div}{\operatorname{div}}
\newcommand{\Hol} {\operatorname{Hol}}
\newcommand{\diam} {\operatorname{diam}}
\newcommand{\Scal} {\operatorname{Scal}}
\newcommand{\scal} {\operatorname{scal}}
\newcommand{\Ric} {\operatorname{Ric}}
\newcommand{\Hess} {\operatorname{Hess}}
\newcommand{\grad} {\operatorname{grad}}
\newcommand{\Sect} {\operatorname{Sect}}
\newcommand{\Rm} {\operatorname{Rm}}
\newcommand{ \Rmzero } {\mathring{\Rm}}
\newcommand{\Rc} {\operatorname{Rc}}
\newcommand{\Curv} {S_{B}^{2}\left( \mathfrak{so}(n) \right) }
\newcommand{ \tr } {\operatorname{tr}}
\newcommand{ \id } {\operatorname{id}}
\newcommand{ \Riczero } {\mathring{\Ric}}
\newcommand{ \ad } {\operatorname{ad}}
\newcommand{ \Ad } {\operatorname{Ad}}
\newcommand{ \dist } {\operatorname{dist}}
\newcommand{ \rank } {\operatorname{rank}}
\newcommand{\Vol}{\operatorname{Vol}}
\newcommand{\dVol}{\operatorname{dVol}}
\newcommand{ \zitieren }[1]{ \hspace{-3mm} \cite{#1}}
\newcommand{ \pr }{\operatorname{pr}}
\newcommand{\diag}{\operatorname{diag}}
\newcommand{\Lagr}{\mathcal{L}}
\newcommand{\av}{\operatorname{av}}
\newcommand{ \floor }[1]{ \lfloor #1 \rfloor }
\newcommand{ \ceil }[1]{ \lceil #1 \rceil }

\newtheorem{theorem}{Theorem}[section]
\newtheorem{definition}[theorem]{Definition}
\newtheorem{example}[theorem]{Example}
\newtheorem{remark}[theorem]{Remark}
\newtheorem{lemma}[theorem]{Lemma}
\newtheorem{proposition}[theorem]{Proposition}
\newtheorem{corollary}[theorem]{Corollary}
\newtheorem{assumption}[theorem]{Assumption}
\newtheorem{acknowledgment}[theorem]{Acknowledgment}
\newtheorem{DefAndLemma}[theorem]{Definition and lemma}

\newenvironment{remarkroman}{\begin{remark} \normalfont }{\end{remark}}
\newenvironment{exampleroman}{\begin{example} \normalfont }{\end{example}}

\newcommand{\R}{\mathbb{R}}
\newcommand{\N}{\mathbb{N}}
\newcommand{\Z}{\mathbb{Z}}
\newcommand{\Q}{\mathbb{Q}}
\newcommand{\C}{\mathbb{C}}
\newcommand{\F}{\mathbb{F}}
\newcommand{\X}{\mathcal{X}}
\newcommand{\D}{\mathcal{D}}
\newcommand{\Cont}{\mathcal{C}}

\renewcommand{\labelenumi}{(\alph{enumi})}
\newtheorem{maintheorem}{Theorem}[]
\renewcommand*{\themaintheorem}{\Alph{maintheorem}}
\newtheorem*{theorem*}{Theorem}
\newtheorem*{corollary*}{Corollary}
\newtheorem*{remark*}{Remark}
\newtheorem*{example*}{Example}
\newtheorem*{question*}{Question}

\begin{abstract}
We prove rigidity results for compact Riemannian manifolds in the spirit of Tachibana. For example, we observe that manifolds with divergence free Weyl tensors and $\floor{\frac{n-1}{2}}$-nonnegative curvature operators are locally symmetric or conformally equivalent to a quotient of the sphere. 

The main focus of the paper is to prove similar results  for manifolds with special holonomy. In particular, we consider K\"ahler manifolds with divergence free Bochner tensors. For quaternion K\"ahler manifolds we obtain a partial result towards the LeBrun-Salamon conjecture. 
\end{abstract}

\maketitle

%\vspace{-6mm} 

\section*{Introduction}

In this paper we establish rigidity theorems for compact Riemannian manifolds. According to a famous theorem of Tachibana \cite{TachibanaPosCurvOperator}, manifolds with harmonic curvature tensors and positive curvature operators have constant sectional curvature. If the curvature operator is nonnegative, then the manifold is locally symmetric. 

In dimension $n=4,$ Micallef-Wang \cite{MicallefWangNIC} proved that a Riemannian manifold with harmonic curvature tensor and nonnegative isotropic curvature is locally symmetric or locally conformally flat. In particular, if the metric is Einstein, then the manifold is locally symmetric.

In the case of Einstein manifolds, the convergence theorems for the Ricci flow due to Hamilton \cite{Hamilton3DimRF, Hamilton4DimRFposCurvOp}, Chen \cite{ChenQuarterPinching}, B\"ohm-Wilking \cite{BW2}, Ni-Wu \cite{NiWuNonnegativeCurvatureOperator}, Brendle-Schoen \cite{BrendleSchoenSphereTheorem, BrendleSchoenWeaklyQuarterPinched}, Brendle \cite{BrendleConvergenceInHihgerDimensions} and Seshadri \cite{SeshadriNIC} imply Tachibana-type theorems. Moreover, Brendle \cite{BrendleEinsteinNIC} proved that Einstein manifolds with nonnegative isotropic curvature are locally symmetric. These results rely on the fact that, e.g., nonnegative curvature operator or nonnegative isotropic curvature are Ricci flow invariant curvature conditions. 

We recall that the curvature operator of a Riemannian manifold is $k$-nonnegative if the sum of its lowest $k$ eigenvalues is nonnegative. 

In \cite{PetersenWinkNewCurvatureConditionsBochner}, the authors proved that Einstein manifolds with $\floor{\frac{n-1}{2}}$-nonnegative curvature operators are locally symmetric. In contrast to the previously mentioned curvature conditions, $\floor{\frac{n-1}{2}}$-nonnegative curvature operator is not preserved by the Ricci flow for $n \geq 7.$

Note that a Riemannian manifold has harmonic curvature tensor if and only if it has constant scalar curvature and the Weyl tensor is divergence free. The main theme of this paper is to show that the assumption of divergence free Weyl tensor is sufficient to prove Tachibana-type theorems. 

In fact, Tran \cite{TranHarmonicWeyl} observed that manifolds with divergence free Weyl tensors and nonnegative curvature operators are locally symmetric or locally conformally flat. Based on the work of Schoen-Yau \cite{SchoenYauConfFlatManifoldsAndScal}, Noronha \cite{NoronhaConfFlatNonnegScal} classified compact locally conformally flat manifolds with nonnegative Ricci curvature. Their universal cover is either conformally equivalent to $S^n$ or isometric to $S^{n-1} \times \R$ or $\R^n.$  

\begin{maintheorem}
Let $(M,g)$ be a compact $n$-dimensional Riemannian manifold with divergence free Weyl tensor. 

If $(M,g)$ has $\floor{\frac{n-1}{2}}$-nonnegative curvature operator, then $(M,g)$ is locally symmetric or conformally equivalent to a quotient of the standard sphere.
\label{MainWeylTheorem}
\end{maintheorem}

For manifolds with special holonomy, the assumption on the eigenvalues of the curvature operator reduces to nonnegative curvature operator. This is immediate from the observation that the curvature operator of a Riemannian manifold vanishes on the complement of the holonomy algebra. 

In order to establish Tachibana-type results for manifolds with special holonomy, it is therefore natural to study the restriction of curvature operator to the holonomy algebra, $\mathfrak{R}_{| \mathfrak{hol}} \colon \mathfrak{hol} \to \mathfrak{hol}.$ 

In case $\mathfrak{hol} = \mathfrak{u}(m),$ this is the {\em K\"ahler curvature operator.} For example, in \cite[Theorem E]{PetersenWinkVanishingHodgeNumbers} the authors proved that a compact K\"ahler-Einstein manifold of real dimension $2m$ is locally symmetric provided the K\"ahler curvature operator is $\floor{\frac{m+1}{2}}$-nonnegative.

Tachibana-type results for K\"ahler manifolds with nonnegative bisectional or nonnegative orthogonal bisectional curvature, respectively, follow from the classification results due to Mori \cite{MoriProjectiveManifoldsWithAmpleTangentBundles}, Siu-Yau \cite{SiuYauCompactKaehlerPosBiholCurv}, Mok-Zhong \cite{MokZhongCurvatureCharacterizationCptHermSymSpaces}, Mok \cite{MokUniformizationKaehler} and Chen \cite{ChenPosOrthBisectCurv}, Gu-Zhang \cite{GuZhangExtensionMokTheorem}. Results for K\"ahler manifolds with nonnegative isotropic curvature were obtained by Seaman \cite{SeamanOnManifoldsWithNIC} and Seshadri \cite{SeshadriNIC}. Note that K\"ahler manifolds with nonnegative isotropic curvature have nonnegative orthogonal bisectional curvature.

The analogue of the Weyl tensor for K\"ahler manifolds is the Bochner tensor, identified by Bochner in \cite{BochnerCurvatureAndBettiNumbersII}. In analogy to the generic holonomy case, a K\"ahler manifold has harmonic curvature tensor if and only if it has constant scalar curvature and the Bochner tensor is divergence free.

Bryant classified compact K\"ahler manifolds with vanishing Bochner tensors in \cite[Corollary 4.17]{BryantBochnerKaehlerMetrics}. In particular, a compact Bochner flat K\"ahler manifold with nonnegative Ricci curvature is isometric to $\mathbb{CP}^m$ or its universal cover is isometric to $\mathbb{C}^m.$ 

Therefore, we have the following generalization of \cite[Theorem E]{PetersenWinkVanishingHodgeNumbers} on K\"ahler-Einstein manifolds:

\begin{maintheorem}
Let $(M,g)$ be a compact K\"ahler manifold of real dimension $2m$ with divergence free Bochner tensor. 

If $(M,g)$ has $\floor{\frac{m+1}{2}}$-nonnegative K\"ahler curvature operator, then $(M,g)$ is locally symmetric.
\label{MainTheoremBochner}
\end{maintheorem}

A Riemannian manifold of real dimension $4m \geq 8$ with holonomy contained in $Sp(m) \cdot Sp(1)$ is called quaternion K\"ahler manifold. If the scalar curvature is positive, then the manifold is called {\em positive quaternion K\"ahler manifold}. 

The LeBrun-Salamon conjecture asserts that every positive quaternion K\"ahler manifold is symmetric. In real dimension $8$ this was proven by Poon-Salamon \cite{PoonSalamonQKEightManifolds} and with different techniques by LeBrun-Salamon \cite{LeBrunSalamonStrongRigidityQK}. If the real dimension of the manifold is $12$ or $16$, then the conjecture follows from the work of Buczy{\'n}ski-Wi\'sniewski \cite{BuczynskiWisniewskiAlgebraicTorusActionsonContactManifolds}. A quaternion K\"ahler manifold of real dimension $4$ is by definition a half conformally flat Einstein manifold. In this case the LeBrun-Salamon conjecture follows from Hitchin's work \cite{HitchinSelfDualityOnRiemannSurface}.

In analogy to the K\"ahler case, for a quaternion K\"ahler manifold we consider the corresponding {\em quaternion K\"ahler curvature operator} by restricting the Riemannian curvature operator to the holonomy algebra $\mathfrak{sp}(m) \oplus \mathfrak{sp}(1).$ 

Notice that in real dimension $4m \geq 8$ quaternion K\"ahler manifolds are necessarily Einstein. In particular, the curvature tensor is automatically harmonic.

Therefore, the analogue of Theorems \ref{MainWeylTheorem} and \ref{MainTheoremBochner} for quaternion K\"ahler manifolds is

\begin{maintheorem}
Let $(M,g)$ be a compact quaternion K\"ahler manifold of real dimension $4m \geq 8.$ 

If $(M,g)$ has $\floor{\frac{m+1}{2}}$-nonnegative quaternion K\"ahler curvature operator, then $(M,g)$ is locally symmetric.
\label{MainQKTheorem}
\end{maintheorem}

Positive quaternion K\"ahler manifolds are necessarily compact and due to a result of Salamon, \cite[Theorem 6.6]{SalamonQKManifolds}, also simply connected. Hence we have the following partial result towards the LeBrun-Salamon conjecture:

\begin{corollary*}
Let $(M,g)$ be a positive quaternion K\"ahler manifold of real dimension $4m \geq 8.$ If $(M,g)$ has $\floor{\frac{m+1}{2}}$-nonnegative quaternion K\"ahler curvature operator, then $(M,g)$ is a symmetric space.
\end{corollary*}

Symmetric quaternion K\"ahler manifolds with positive scalar curvature are classified by Wolf \cite{WolfQKSymmetricSpaces}. In particular, to identify $(M,g)$ isometrically as $\mathbb{HP}^m,$ in addition to being $\floor{\frac{m+1}{2}}$-nonnegative, the quaternion K\"ahler curvature operator only needs to be $k(m)$-positive for a function $k(m) \sim m^2.$ 

We note that Amann \cite{AmannPartialClassificationPositiveQK} proved that positive quaternion K\"ahler manifolds are symmetric provided the dimension of the isometry group is large. Other partial results towards the LeBrun-Salamon conjecture have been obtained by, e.g., Amann \cite{AmannQKmanifoldsAndBFourEqualOne}, Berger \cite{BergerTroisRemarquesVarieteCourburePositive}, Buczy{\'n}ski-Wi\'sniewski \cite{BuczynskiWisniewskiAlgebraicTorusActionsonContactManifolds}, Chow-Yang \cite{ChowYangRigidityNonnegCCurevedQKManifolds}, Fang \cite{FangPositiveQKandSymmetry}, LeBrun \cite{LeBrunFanoManifoldsContactStructuresAndQK}, Occhetta-Romano-Conde-Wi\'{s}niewski \cite{OccettaEtAlHighRankTorusActionsOnContactManifolds}, Salamon \cite{SalamonIndexTheoryandQKmanifolds, SalamonQKSurvey} and Semmelmann-Weingart \cite{SemmelmannWeingartHilbertPolynominalQK}.

\vspace{2mm}

Recall that the remaining holonomy groups in Berger's list of irreducible holomony groups force the metric to be either locally symmetric or Ricci flat. Notice that a Ricci flat manifold whose curvature operator satisfies one of the nonnegativity assumptions in Theorems \ref{MainWeylTheorem} - \ref{MainQKTheorem} is flat. 

\vspace{2mm}

The proofs of Theorems \ref{MainWeylTheorem} - \ref{MainQKTheorem} rely on the Bochner technique and the fact that the Lichnerowicz Laplacian preserves tensor bundles which are invariant under the holonomy representation. In particular, if $R$ is a harmonic curvature tensor on $(M,g)$ and $\mathfrak{R}$ is the curvature operator of $(M,g)$, then $R$ satisfies the Bochner identity
\begin{align*}
\Delta \frac{1}{2} |R|^2 = | \nabla R |^2 + \frac{1}{2} \cdot g( \mathfrak{R}(R^{\mathfrak{hol}}), R^{\mathfrak{hol}}),
\end{align*}
where the curvature term $g( \mathfrak{R}(R^{\mathfrak{hol}}), R^{\mathfrak{hol}})$ is adapted to the holonomy algebra $\mathfrak{hol}.$ Thus, if $g( \mathfrak{R}(R^{\mathfrak{hol}}), R^{\mathfrak{hol}}) \geq 0,$ then the manifold is locally symmetric. 

For example, if $\mathfrak{R}_{| \mathfrak{hol}} \colon \mathfrak{hol} \to \mathfrak{hol}$ is $2$-nonnegative, then $g( \mathfrak{R}(R^{\mathfrak{hol}}), R^{\mathfrak{hol}}) \geq 0$ according to proposition \ref{CurvatureTermFor2NonnegativeCurv}. This is a useful observation in low dimensions.

Furthermore, corollaries \ref{BochnerForWeyl},  \ref{BochnerTechniqueForBochnerTensor} and proposition \ref{BochnerTechniqueQKmaniolds} show that $g( \mathfrak{R}(R^{\mathfrak{hol}}), R^{\mathfrak{hol}}) \geq 0$ provided a weighted sum of eigenvalues of the curvature operator is nonnegative. In particular, Theorems \ref{MainWeylTheorem} - \ref{MainQKTheorem} generalize to these weighted curvature conditions. \vspace{2mm}

The paper is structured as follows: Section \ref{SectionPreliminaries} briefly reviews the relevant details of the Bochner technique. Section \ref{SectionWeylTensors} proves Theorem \ref{MainWeylTheorem} by combining corollary \ref{BochnerForWeyl} with results from the literature. In section \ref{SectionBochnerTensor} we show that K\"ahler manifolds with harmonic Bochner tensors are Bochner flat or have constant scalar curvature, and deduce Theorem \ref{MainTheoremBochner}. Finally, in section \ref{SectionQKmanifolds} we compute the curvature term of the Lichnerowicz Laplacian for quaternion K\"ahler curvature tensors and prove Theorem \ref{MainQKTheorem}. \vspace{2mm}

\textit{Acknowledgements.} We would like to thank Xiaolong Li for communications on compact Riemannian manifolds with divergence free Weyl tensors and $\floor{\frac{n-1}{2}}$-nonnegative curvature operators. In \cite[Corollary 3.3]{PetersenWinkNewCurvatureConditionsBochnerArxiv} we proved that in this case the Weyl tensor is parallel. However, this result was not published in \cite{PetersenWinkNewCurvatureConditionsBochner}. Xiaolong Li gave an independent proof and also observed that the manifolds are locally symmetric or locally conformally flat.

We would like to thank the referee for helpful comments.

%MW is funded by the Deutsche Forschungsgemeinschaft (DFG, German Research Foundation) under Germany's Excellence Strategy EXC 2044–390685587, Mathematics M\"unster: Dynamics–Geometry–Structure.

\section{Preliminaries}
\label{SectionPreliminaries}

We summarize the relevant material from \cite[Section 1]{PetersenWinkVanishingHodgeNumbers} and focus on the Bochner technique for curvature tensors.

\subsection{Tensors}
Let $(V,g)$ be an $n$-dimensional Euclidean vector space. The metric $g$ induces a metric on $\bigotimes^{k} V^{*}$ and $\Ext^k V$. In particular, if $\lbrace e_i \rbrace_{i=1, \ldots, n}$ is an orthonormal basis for $V$, then $\left\lbrace  e_{i_1} \wedge \ldots \wedge e_{i_k} \right\rbrace _{1 \leq i_1 < \ldots < i_k \leq n}$ is an orthonormal basis for $\Ext^k V.$ 

Notice that $\Ext^2 V$ inherits a Lie algebra structure from $\mathfrak{so}(V)$. The induced Lie algebra action on $V$ is given by
\begin{align*}
(X \wedge Y) Z = g(X,Z)Y - g(Y,Z)X.
\end{align*}
In particular, for $\Xi_{\alpha}, \Xi_{\beta} \in \Ext^2 V$ we have 
\begin{align*}
(\Xi_{\alpha}) \Xi_{\beta} = [ \Xi_{\alpha}, \Xi_{\beta}].
\end{align*}
Similarly, for $T \in \bigotimes^k V^{*}$ and $L \in \mathfrak{so}(V)$ set
\begin{align*}
(LT)(X_1, \ldots, X_k) = - \sum_{i=1}^k T(X_1, \ldots, LX_i, \ldots, X_k).
\end{align*}

A tensor $\Rm \in \bigotimes^4 V^{*}$ is an algebraic curvature tensor if 
\begin{align*}
\Rm(X,Y, & Z,W) =-\Rm(Y,X,Z,W)=-\Rm(X,Y,W,Z)=\Rm(Z,W,X,Y), \\
& \Rm(X,Y,Z,W) + \Rm(Y,Z,X,W)+ \Rm(Z,X,Y,W)=0.
\end{align*}
In particular, it induces the curvature operator $\mathfrak{R} \colon \Ext^2 V \to \Ext^2 V$ via
\begin{align*}
g( \mathfrak{R}(X \wedge Y), Z \wedge W ) = \Rm(X,Y,Z,W).
\end{align*}
The associated symmetric bilinear form is denoted by $R \in \operatorname{Sym}^2_B \left( \Ext^2 V \right).$ Notice that 
\begin{align*}
| \Rm |^2 = 4 |R|^2.
\end{align*}

\begin{example} \normalfont
For $S,T \in \bigotimes^2 V^*$ set 
\begin{align*}
(S \owedge T)(X,Y,Z,W) = & \ S(X,Z)T(Y,W)-S(X,W)T(Y,Z) \\
& +S(Y,W)T(X,Z)-S(Y,Z)T(X,W).
\end{align*}
In particular, $g \owedge g$ is the curvature tensor of the sphere of radius $1 / \sqrt{2}.$
\end{example}

\begin{remark} \normalfont
\label{CurvatureOperatorRestrictsToHolonomy}
The curvature operator $\mathfrak{R}$ of a Riemannian manifold $(M,g)$ vanishes on the complement of the holonomy algebra $\mathfrak{hol}.$ In particular, it induces $\mathfrak{R}_{| \mathfrak{hol}} \colon \mathfrak{hol} \to \mathfrak{hol}$ and the corresponding curvature tensor $R \in \operatorname{Sym}^2_B \left( \mathfrak{hol} \right).$

If $\mathfrak{hol}=\mathfrak{u}(m),$ then $(M,g)$ is K\"ahler. The operator $\mathfrak{R}_{| \mathfrak{u}(m)} \colon \mathfrak{u}(m) \to \mathfrak{u}(m)$ is called {\em K\"ahler curvature operator} and  the associated $R \in \operatorname{Sym}^2_B \left( \mathfrak{u}(m) \right)$ is the {\em K\"ahler curvature tensor}.

If $\mathfrak{hol}=\mathfrak{sp}(m) \oplus \mathfrak{sp}(1),$ then $(M,g)$ is a quaternion K\"ahler manifold. The operator $\mathfrak{R}_{| \mathfrak{sp}(m) \oplus \mathfrak{sp}(1)} \colon \mathfrak{sp}(m) \oplus \mathfrak{sp}(1) \to \mathfrak{sp}(m) \oplus \mathfrak{sp}(1)$ is called {\em quaternion K\"ahler curvature operator} and  the associated $R \in \operatorname{Sym}^2_B \left( \mathfrak{sp}(m) \oplus \mathfrak{sp}(1) \right)$ is the {\em quaternion K\"ahler curvature tensor}.\end{remark}

If $\mathfrak{g} \subset \mathfrak{so}(V)$ is a Lie subalgebra, define $T^{\mathfrak{g}} \in \left( \bigotimes^k V^{*} \right) \otimes_{\R} \mathfrak{g} $ by 
\begin{equation*}
g( L, T^{\mathfrak{g}}(X_1, \ldots, X_k)) = (LT)(X_1, \ldots, X_k)
\end{equation*}
for all $L \in \mathfrak{g} \subset \mathfrak{so}(V) = \Ext^2V$.
If $\lbrace \Xi_{\alpha} \rbrace$ is an orthonormal eigenbasis for $\mathfrak{R} \colon \mathfrak{g} \to \mathfrak{g}$, then 
\begin{align*}
\mathfrak{R}(T^{\mathfrak{g}}) = \mathfrak{R} \circ T^{\mathfrak{g}} = \sum \mathfrak{R}( \Xi_{\alpha} ) \otimes \Xi_{\alpha} T.
\end{align*}
In particular, if $\lbrace \lambda_{\alpha} \rbrace$ denote the corresponding eigenvalues, then
\begin{align*}
g( \mathfrak{R}(T^{\mathfrak{g}}), T^{\mathfrak{g}}) 
= \sum \lambda_{\alpha} | \Xi_{\alpha} T |^2 \ \text{ and } \ |T^{\mathfrak{g}}|^2 = \sum | \Xi_{\alpha} T |^2.
\end{align*} 

\begin{proposition}
Let $R \in \operatorname{Sym}_B^2( \mathfrak{g})$ be an algebraic curvature tensor and let $\mathfrak{R} \colon \mathfrak{g} \to \mathfrak{g}$ be the corresponding curvature operator. 

If $\mathfrak{R}$ is $2$-nonnegative, then 
\begin{align*}
g( \mathfrak{R}(R^{\mathfrak{g}}), R^{\mathfrak{g}}) \geq 0.
\end{align*}
\label{CurvatureTermFor2NonnegativeCurv}
\end{proposition}
\begin{proof}
Let $\lbrace \Xi_{\alpha}\rbrace$ denote an orthonormal eigenbasis of $\mathfrak{R}$ and let $\lambda_1 \leq \ldots \leq \lambda_{\dim \mathfrak{g}}$ denote the corresponding eigenvalues. \cite[Example 1.2]{PetersenWinkVanishingHodgeNumbers}  shows that
\begin{align*}
| \Xi_{\gamma} R |^2 
= 2 \sum_{\alpha < \beta} ( \lambda_{\alpha} - \lambda_{\beta} )^2 g( ( \Xi_{\gamma} ) \Xi_{\alpha}, \Xi_{\beta} )^2.
\end{align*}
Recall that $( \Xi_{\alpha} ) \Xi_{\beta} = [ \Xi_{\alpha}, \Xi_{\beta} ]$ and thus $g( ( \Xi_{\alpha} ) \Xi_{\beta}, \Xi_{\gamma} )^2$ is fully symmetric. Therefore
\begin{align*}
g( \mathfrak{R}(R^{\mathfrak{g}}), R^{\mathfrak{g}})   =  2 \sum_{\gamma} \sum_{\alpha < \beta} \lambda_{\gamma} ( \lambda_{\alpha} - \lambda_{\beta} )^2 g( ( \Xi_{\gamma} ) \Xi_{\alpha}, \Xi_{\beta} )^2 
=  2 \sum_{
I
%\begin{tabular}{c}
%\tiny
%$I = \lbrace \alpha, \beta, \gamma \rbrace,$ $|I|=3$ \\
%\tiny
%$I \subset \lbrace 1, \ldots, \dim \mathfrak{g} \rbrace $
%\end{tabular}
}
\Lambda_{\alpha \beta \gamma} g( ( \Xi_{\alpha} ) \Xi_{\beta}, \Xi_{\gamma} )^2,
\end{align*}
where the index set $I = \lbrace \alpha, \beta, \gamma \rbrace$ satisfies $I \subset \lbrace 1, \ldots, \dim \mathfrak{g} \rbrace$, $|I|=3$ and 
\begin{align*}
\Lambda_{\alpha \beta \gamma} = \lambda_{\alpha}  ( \lambda_{\beta} - \lambda_{\gamma} )^2 + \lambda_{\beta}  ( \lambda_{\gamma} - \lambda_{\alpha} )^2 + \lambda_{\gamma} ( \lambda_{\alpha} - \lambda_{\beta} )^2.
\end{align*}
We may assume $\alpha < \beta < \gamma.$ Since $\mathfrak{R}$ is $2$-nonnegative,  we have $\Lambda_{\alpha \beta \gamma} \geq 0$ if $\alpha \geq 2.$ Thus the claim follows from 
\begin{align*}
\Lambda_{1 \beta \gamma} & =  \lambda_{1}  ( \lambda_{\beta} - \lambda_{\gamma} )^2 + \lambda_{\beta}  ( \lambda_{1} - \lambda_{\gamma} )^2 + \lambda_{\gamma} ( \lambda_{1} - \lambda_{\beta} )^2 \\
& \geq ( \lambda_1 + \lambda_{\beta} ) ( \lambda_1 - \lambda_{\gamma} )^2 +  \lambda_{\gamma} ( \lambda_{1} - \lambda_{\beta} )^2  \geq 0.
\end{align*}
\end{proof}

\subsection{The Bochner technique} Let $(M,g)$ be an $n$-dimensional Riemannian manifold and let $R(X,Y)Z = \nabla_Y \nabla_X Z - \nabla_X \nabla_Y Z + \nabla_{[X,Y]} Z$ denote its curvature tensor. For a $(0,k)$-tensor $T$ set 
\begin{equation*}
\Ric(T)(X_1, \ldots, X_k) = \sum_{i=1}^k \sum_{j=1}^n  (R(X_i,e_j)T) (X_1, \ldots, e_j, \ldots, X_k),
\end{equation*}
where $e_1, \dots, e_n$ is a local orthonormal frame and 
\begin{align*}
R(X,Y) T(X_1, \ldots, X_k) = - \sum_{i=1}^k T(X_1, \ldots, R(X,Y) X_i, \ldots, X_k)
\end{align*}
according to the Ricci identity. 

The divergence of $T$ is given by
\begin{align*}
( \Div T ) (X_1, \ldots, X_{k-1}) = - (\nabla^{*}T)(X_1, \ldots, X_{k-1}) = \sum_{i=1}^n ( \nabla_{e_i} T ) ( e_i , X_1, \ldots, X_{k-1}).
\end{align*} 

In this paper we will focus on algebraic curvature tensors on Riemannian manifolds. Notice that the proof of \cite[Theorem 9.4.2]{PetersenRiemGeom} also shows

\begin{proposition}
\label{SecondBianchiDivergenceFreeImpliesHarmonic}
Let $(M,g)$ be a Riemannian manifold. Suppose that $T$ is an algebraic curvature tensor on $M$, i.e. $T$ satisfies
\begin{align*}
T(X,Y, & Z,W) =-T(Y,X,Z,W)=-T(X,Y,W,Z)=T(Z,W,X,Y), \\
& T(X,Y,Z,W) + T(Y,Z,X,W)+ T(Z,X,Y,W)=0.
\end{align*}

If in addition $T$ satisfies the second Bianchi identity and $T$ is divergence free, then $T$ is harmonic, 
\begin{align*}
\nabla^{*} \nabla T + \frac{1}{2} \cdot \Ric(T) = 0.
\end{align*}
\end{proposition}

A curvature tensor $R \in \operatorname{Sym}_B^2(TM)$ is called harmonic if the corresponding $(0,4)$-curvature tensor $\Rm$ is harmonic.

\begin{corollary}
Let $(M,g)$ be a Riemannian manifold. Let $\mathfrak{R} \colon \Ext^2 TM \to \Ext^2 TM$ denote its curvature operator and $\mathfrak{hol}$ its holonomy algebra. If $T$ is a harmonic curvature tensor on $M,$ then 
\begin{align*}
\Delta \frac{1}{2} |T|^2 = | \nabla T |^2 + \frac{1}{2} \cdot g( \mathfrak{R}(T^{\mathfrak{hol}}),T^{\mathfrak{hol}}) = 0.
\end{align*}
In particular, if in addition $M$ is compact and $g( \mathfrak{R}(T^{\mathfrak{hol}}),T^{\mathfrak{hol}} ) \geq 0,$ then $T$ is parallel. 
\label{BochnerTechnique}
\end{corollary}
\begin{proof}
According to \cite[Proposition 1.6]{PetersenWinkVanishingHodgeNumbers} the curvature term in the Bochner formula can be computed by
\begin{align*}
g( \Ric(T),T) = g( \mathfrak{R}(T^{\mathfrak{hol}}),T^{\mathfrak{hol}} ).
\end{align*}
Thus the claim follows from proposition \ref{SecondBianchiDivergenceFreeImpliesHarmonic} and the maximum principle.
\end{proof}

A general criterion to show $g( \mathfrak{R}(T^{\mathfrak{hol}}),T^{\mathfrak{hol}}) \geq 0$ based on the eigenvalues of the curvature operator $\mathfrak{R}_{| \mathfrak{hol}} \colon \mathfrak{hol} \to \mathfrak{hol}$ is established in \cite[Lemma 1.8]{PetersenWinkVanishingHodgeNumbers}. As an application thereof, proposition \ref{CurvatureTermFor2NonnegativeCurv} shows that if $\mathfrak{R}_{| \mathfrak{hol}} \colon \mathfrak{hol} \to \mathfrak{hol}$ is $2$-nonnegative and $R \in \operatorname{Sym}_B^2( \mathfrak{hol})$ is the associated curvature tensor, then $g( \mathfrak{R}(R^{\mathfrak{hol}}),R^{\mathfrak{hol}}) \geq 0.$ Thus we have

\begin{corollary}
Let $(M,g)$ be a compact Riemannian manifold with holonomy algebra $\mathfrak{hol}.$ If the curvature operator $\mathfrak{R}_{| \mathfrak{hol}} \colon \mathfrak{hol} \to \mathfrak{hol}$ is $2$-nonnegative, then $(M,g)$ is locally symmetric.
\end{corollary}

\section{Manifolds with divergence free Weyl tensors}
\label{SectionWeylTensors}

Let $(M,g)$ be a compact $n$-dimensional Riemannian manifold. The decomposition of the space of curvature tensors into orthogonal, irreducible, $O(n)$-invariant modules yields
\begin{align*}
\Rm = \frac{\scal}{2(n-1)n} g \owedge g + \frac{1}{n-2} g \owedge \mathring{\Ric} + W,
\end{align*}
where $\Riczero = \Ric - \frac{\scal}{n} g$ denotes the trace-free Ricci tensor and $W$ the Weyl tensor.

\begin{remark} \normalfont
\label{CurvatureTensorHarmonic}
Recall that the curvature tensor of a Riemannian manifold is divergence free if and only if the Weyl tensor is divergence free and the scalar curvature is constant, since
\begin{align*}
(\Div \Rm)(Z,X,Y)  = & \ (\nabla_X \Ric)(Y,Z)  - (\nabla_Y \Ric)(X,Z)\\
= & - \frac{1}{2(n-1)} d \scal( (X \wedge Y) Z ) + \frac{n-2}{n-3} \Div W(Z,X,Y).
\end{align*} 
\end{remark}

\begin{proposition}
\label{BianchiForWeyl}
Let $(M,g)$ be a Riemannian manifold. If the Weyl curvature $W$ is divergence free, then $W$ satisfies the second Bianchi identity and 
\begin{align*}
\nabla^{*} \nabla W + \frac{1}{2} \Ric(W) = 0.
\end{align*}
\end{proposition}
\begin{proof}
The fact that divergence free Weyl tensors satisfy the second Bianchi identity is explained in \cite[section 28]{EisenhartRiemannianGeometry}. The Bochner formula follows from proposition \ref{SecondBianchiDivergenceFreeImpliesHarmonic}.
\end{proof}

\begin{corollary}
Let $(M,g)$ be a compact $n$-dimensional Riemannian manifold. Suppose that the Weyl tensor is divergence free. If the eigenvalues $\lambda_1 \leq \ldots \leq \lambda_{\binom{n}{2}}$ of the curvature operator satisfy 
\begin{align*}
 \lambda_1 + \ldots + \lambda_{\floor{\frac{n-1}{2}}} + \frac{1+(-1)^n}{4} \lambda_{\floor{\frac{n-1}{2}}+1} \geq 0   \ \text{for} \ n \geq 4,
\end{align*}
then the Weyl tensor is parallel. Moreover, if the inequality is strict, then $(M,g)$ is locally conformally flat.
\label{BochnerForWeyl}
\end{corollary}
\begin{proof}
\cite[Lemma 2.2 and Proposition 2.5]{PetersenWinkNewCurvatureConditionsBochner} imply that 
\begin{align*}
|L W |^2 \leq 8 | W |^2 |L|^2 = \frac{2}{n-1} | W^{\mathfrak{so}} |^2 | L|^2
\end{align*}
for all $L \in \mathfrak{so}(TM).$ Thus the curvature assumption shows that $g(\mathfrak{R}(W^{\mathfrak{so}}),W^{\mathfrak{so}}) \geq 0$ due to \cite[Lemma 1.8]{PetersenWinkVanishingHodgeNumbers}. Proposition \ref{BianchiForWeyl} and the Bochner technique as in corollary \ref{BochnerTechnique} imply that the Weyl tensor is parallel. 

Moreover, if $\lambda_1 + \ldots + \lambda_{\floor{\frac{n-1}{2}}} + \frac{1+(-1)^n}{4} \lambda_{\floor{\frac{n-1}{2}}+1} > 0$, then \cite[Lemma 1.8]{PetersenWinkVanishingHodgeNumbers} shows that $W^{\mathfrak{so}}=0$ and thus $W=0$ due to \cite[Proposition 2.5]{PetersenWinkNewCurvatureConditionsBochner}.
\end{proof}

\begin{proposition}[G{\l}odek] 
Let $(M,g)$ be a Riemannian manifold. If the Weyl tensor is parallel, then $(M,g)$ has constant scalar curvature or $(M,g)$ is conformally flat.  
\label{ParallelWeylImpliesConstantScalorWZero}
\end{proposition}
\begin{proof}
This was established by G{\l}odek in \cite{GlodekConformallySymmetricSpaces}. We include a modified proof to illustrate the idea behind the proof of proposition \ref{ParallelBochnerImpliesConstantScalorBZero}, the K\"ahler analogue of proposition \ref{ParallelWeylImpliesConstantScalorWZero}.

It follows from remark \ref{CurvatureTensorHarmonic} and $\nabla W= 0$ that
\begin{align*}
\sum_{i=1}^n \left( \nabla_{e_i} \Rm \right)(e_i, Z,X,Y)  =  \frac{1}{2(n-1)} \left( d \scal(X) g(Y,Z) - d \scal(Y) g(X,Z) \right).
\end{align*} 
Thus we have
\begin{align*}
\sum_{i=1}^n \left( \nabla_{e_i} R \right) (e_i, Z) X = \frac{1}{2(n-1)} \left( d \scal(X) Z - \nabla \scal g(X,Z) \right),
\end{align*}
where $R$ denotes the $(1,3)$-curvature tensor. 

Consider the Lie algebra action of the curvature tensor $R(X,Y) \in \mathfrak{so}(TM)$ on the Weyl tensor. Since $\nabla W =0$ we have $(R(X,Y))W=0$ and consequently $\left( \left( \nabla_Z R \right)(X,Y) \right) W = 0.$ 

Overall we obtain
\begin{align*}
0  = & \ 2(n-1) \sum_{i=1}^n \left( \left( \left( \nabla_{e_i} R \right) (e_i, Z) \right)W \right) (E_1, E_2, E_3, E_4) \\
= & \ g(Z,E_1 ) W( \nabla \scal, E_2, E_3, E_4 ) - d \scal (E_1 ) W( Z, E_2, E_3, E_4 ) \\
& \ +  g(Z,E_2 ) W( E_1, \nabla \scal, E_3, E_4 ) - d \scal (E_2 ) W( E_1, Z, E_3, E_4 ) \\
& \ + g(Z,E_3 ) W( E_1, E_2, \nabla \scal, E_4 ) - d \scal (E_3 ) W( E_1, E_2, Z, E_4 ) \\
& \ + g(Z,E_4 ) W( E_1, E_2, E_3, \nabla \scal ) - d \scal (E_4 ) W( E_1, E_2, E_3, Z ).
\end{align*}
Contraction of $E_1$ with $Z$ yields 
\begin{align*}
0 = (n-1) \cdot  W( \nabla \scal, E_2, E_3, E_4 )
\end{align*}
since $W$ is totally trace free and satisfies the algebraic Bianchi identity. Inserting this equation back into the equation above and setting $E_1 = \nabla \scal$ implies
\begin{align*}
| \nabla \scal |^2 \cdot W = 0
\end{align*}
and the claim follows.
\end{proof}

\textit{Proof} of Theorem \ref{MainWeylTheorem}. The assumptions in Theorem \ref{MainWeylTheorem} and corollary \ref{BochnerForWeyl} imply that the Weyl tensor is parallel. Thus G{\l}odek's work \cite{GlodekConformallySymmetricSpaces} shows that $(M,g)$ is conformally flat or has constant scalar curvature. 

If the scalar curvature is constant, then a result of Derdzi{\'n}ski-Roter \cite{DerdzinskiRoterConfSymMetricOfIndices}, see also  Roter \cite{RoterConformallySymmetricSpacesWithDefiniteMetrics}, shows that the Ricci tensor is parallel. Hence the curvature tensor is parallel and $(M,g)$ is locally symmetric. 

If the manifold is conformally flat, then the classification of compact conformally flat manifolds with nonnegative Ricci curvature due to Noronha \cite{NoronhaConfFlatNonnegScal} implies that $(M,g)$ is locally symmetric or conformally equivalent to a quotient of the sphere. $\hfill \Box$

\section{K\"ahler manifolds with divergence free Bochner tensors}
\label{SectionBochnerTensor}

Let $(M,J,g)$ be a K\"ahler manifold of real dimension $2m.$ Let $\omega(X,Y)=g(JX,Y)$ denote the K\"ahler form and $\rho(X,Y)=\Ric( JX, Y )$ denote the Ricci form. The trace-free Ricci tensor is $\mathring{\Ric}=\Ric - \frac{\scal}{2m} g$ and the primitive part of the Ricci form is $\rho_0= \rho - \frac{\scal}{2m} \omega$.

The curvature tensor decomposes into a K\"ahler curvature tensor with constant holomorphic sectional curvature, a K\"ahler curvature tensor with trace-free Ricci curvature and the Bochner tensor,
\begin{align*}
\Rm = & \ \frac{\scal}{4m(m+1)} \left(  \frac{1}{2} g \owedge g + \frac{1}{2} \omega \owedge \omega + 2 \omega \otimes \omega \right) \\
& \ + \frac{1}{2(m+2)} \left( \mathring{\Ric} \owedge g + \rho_0 \owedge \omega + 2 \left( \rho_0 \otimes \omega + \omega \otimes \rho_0 \right) \right) + B.
\end{align*}

The tensor $B$ was introduced by Bochner in  \cite{BochnerCurvatureAndBettiNumbersII} as the analogue of the Weyl tensor.  Alekseevski \cite{AlekseevskiiRiemannianSpacesExceptionalHol} observed that this is indeed the decomposition of a K\"ahler curvature tensor according to decomposition of the space of K\"ahler curvature tensors into orthogonal, $U(m)$-invariant, irreducible subspaces. 

Hence the Bochner tensor satisfies 
\begin{comment}
\begin{align*}
B = & \ \Rm - \frac{1}{2(m+2)} \left( \Ric \owedge g + \rho \owedge \omega + 2 \left( \rho \otimes \omega + \omega \otimes \rho \right) \right) \\
& + \frac{\scal}{4(m+1)(m+2)} \left( \frac{1}{2} g \owedge g + \frac{1}{2} \omega \owedge \omega + 2 \omega \otimes \omega \right),
\end{align*}
that is
\end{comment}
\begin{align*}
B(X,Y,Z,W) = & \ \Rm (X, Y, Z, W)  \\
& \ - \frac{1}{2(m+2)} \left( \Ric(X,Z)g(Y,W) - \Ric(X,W)g(Y,Z) \right. \\ 
& \hspace{26mm} + g(X,Z) \Ric(Y,W) - g(X,W) \Ric(Y,Z) \\
&  \hspace{26mm} + \Ric(JX,Z)g(JY,W) - \Ric(JX,W)g(JY,Z) \\ 
& \hspace{26mm} + g(JX,Z) \Ric(JY,W) - g(JX,W) \Ric(JY,Z) \\
&  \hspace{26mm} \left. + 2\Ric(JX,Y) g(JZ,W) +2 g(JX,Y) \Ric(JZ,W) \right) \\
& \ + \frac{\scal}{4(m+1)(m+2)} \left( g(X,Z)g(Y,W) - g(X,W)g(Y,Z) \right. \\ 
& \hspace{40mm} + g(JX,Z)g(JY,W) - g(JX,W)g(JY,Z) \\ 
& \hspace{40mm} \left. + 2 g(JX,Y)g(JZ,W) \right).
\end{align*}

\begin{remark} \normalfont
\label{BochnerTotallyTraceFree}
Recall that every K\"ahler curvature tensor satisfies
\begin{align*}
\Rm (X,Y,Z,W) = \Rm (JX,JY,Z,W) = \Rm (X,Y,JZ,JW).
\end{align*}
In addition, Tachibana \cite{TachibanaBochnerTensor} proved that the Bochner tensor is totally trace-free. That is, if $e_1, \ldots, e_{2m}$ is an orthonormal basis of $TM,$ then 
\begin{align*}
\sum_{i=1}^{2m} B(e_i,Y,e_i,W) = \sum_{i=1}^{2m} B(e_i, Je_i ,Z,W) = 0.
\end{align*}
\end{remark}

It is straightforward to compute that
\begin{align*}
\Div B(Y,Z,W) = & \ \frac{m}{m+2} \left\lbrace  \left( \nabla_Z \Ric \right)(W,Y) - \left( \nabla_W \Ric \right)(Z,Y) \right. \\
& \hspace{18mm} + \frac{1}{4(m+1)} \left( d \scal (W) g(Z,Y) - d \scal (Z) g(W,Y) \right. \\
& \hspace{42mm} + d \scal (JW) g( JZ,Y ) - d \scal (JZ) g (JW,Y ) \\
& \hspace{42mm} + \left. \left. 2 d \scal (JY) g(JZ,W) \right) \right\rbrace.
\end{align*}

\begin{proposition}
Let $(M,g)$ be a K\"ahler manifold. If the Bochner tensor is divergence free, then it satisfies the second Bianchi identity and consequently
\begin{align*}
\nabla^{*} \nabla B + \frac{1}{2} \Ric(B) = 0.
\end{align*}
\label{BianchiIdentityBochnerTensor}
\end{proposition}
\begin{proof}
The fact that divergence free Bochner tensors satisfy the second Bianchi identity is a result of Omachi \cite{OmachiBianchiIdentityBochnerTensor}. Proposition \ref{SecondBianchiDivergenceFreeImpliesHarmonic} shows that hence the Bochner tensor is harmonic.
\end{proof}

Recall from remark \ref{CurvatureOperatorRestrictsToHolonomy} that the K\"ahler curvature operator is the restriction of the Riemannian curvature operator to the holonomy algebra $\mathfrak{u}(m).$ 

\begin{corollary}
Let $(M,g)$ be a compact K\"ahler manifold of real dimension $2m.$ Suppose that the Bochner tensor is divergence free. If the eigenvalues $\mu_1 \leq \ldots \leq \mu_{m^2}$ of the K\"ahler curvature operator satisfy
\begin{align*}
\mu_1 + \ldots + \mu_{\floor{\frac{m+1}{2}}} + \frac{1+(-1)^m}{4} \mu_{\floor{\frac{m+1}{2}}+1} \geq 0,
\end{align*}
then the Bochner tensor is parallel. Moreover, if the inequality is strict, then $(M,g)$ is Bochner flat.
\label{BochnerTechniqueForBochnerTensor}
\end{corollary}
\begin{proof}
 \cite[Lemma 2.2]{PetersenWinkNewCurvatureConditionsBochner} and \cite[Lemma 5.2]{PetersenWinkVanishingHodgeNumbers} imply that 
\begin{align*}
|L B |^2 \leq 8 | B |^2 |L|^2 = \frac{2}{m+1} | B^{\mathfrak{u}} |^2 | L|^2
\end{align*}
for all $L \in \mathfrak{u}(TM).$ Thus the assumption on the eigenvalues of the K\"ahler curvature operator implies $g(\mathfrak{R}(B^{\mathfrak{u}}),B^{\mathfrak{u}}) \geq 0$ due to \cite[Lemma 1.8]{PetersenWinkVanishingHodgeNumbers}. Proposition \ref{BianchiIdentityBochnerTensor} and the Bochner technique as in corollary \ref{BochnerTechnique} show that $\nabla B=0.$

Moreover, if $\mu_1 + \ldots + \mu_{\floor{\frac{m+1}{2}}} + \frac{1+(-1)^m}{4} \mu_{\floor{\frac{m+1}{2}}+1}>0,$ then \cite[Lemma 1.8]{PetersenWinkVanishingHodgeNumbers} implies that $B^{\mathfrak{u}}=0$. In fact we have $B=0$ due to \cite[Lemma 5.2]{PetersenWinkVanishingHodgeNumbers}.
\end{proof}

\begin{proposition}
Let $(M,g)$ be a K\"ahler manifold. If the Bochner tensor is parallel, then $(M,g)$ has constant scalar curvature or $(M,g)$ is Bochner flat. 
\label{ParallelBochnerImpliesConstantScalorBZero}
\end{proposition}
\begin{proof}
If $\Div B =0,$ then 
\begin{align*}
4(m+1) ( \Div \Rm ) (Y,Z,W) = & \ 4(m+1) ( (\nabla_Z \Ric)(W,Y) - (\nabla_W \Ric)(Z,Y) ) \\
= & \ d \scal (Z) g(W,Y) - d \scal (W) g(Z,Y)  \\
& \ + d \scal (JZ) g (JW,Y ) -d \scal (JW) g( JZ,Y )  \\
& \ - 2 d \scal (JY) g(JZ,W).
\end{align*}

Note that $d \scal( JW ) = g ( \nabla \scal, JW ) = - g ( J \nabla \scal, W).$ Thus if $e_1, \ldots, e_{2m}$ is an orthonormal basis of $TM,$ then
\begin{align*}
4(m+1) \sum_{i=1}^{2m} \left( (\nabla_{e_i} R)(e_i,Y) \right) Z = & \ d \scal (Z) Y -  g(Z,Y) \nabla \scal \\
& \ - d \scal (JZ) JY +  g (JZ,Y) J \nabla \scal  - 2 d \scal (JY) JZ.
\end{align*}

Now suppose that $\nabla B =0$. This implies $R(X,Y)B=0$ and thus $(( \nabla_Z R)(X,Y))B = 0.$ In particular, 
\begin{align*}
0 = & \ - 4(m+1) \sum_{i=1}^{2m} ( ( (\nabla_{e_i} R)(e_i,Y)  )B )(E_1, E_2, E_3, E_4) \\
= & \  4(m+1) \sum_{i=1}^{2m} \lbrace B( ( (\nabla_{e_i} R)(e_i,Y)  )E_1, E_2, E_3, E_4) +  B(E_1, ( (\nabla_{e_i} R)(e_i,Y)  )E_2, E_3, E_4)   \\
& \hspace{26mm} +  B(E_1, E_2, ( (\nabla_{e_i} R)(e_i,Y)  )E_3, E_4) + B(E_1, E_2, E_3, ( (\nabla_{e_i} R)(e_i,Y)  )E_4) \rbrace \\[2mm]
= & \ d \scal( E_1 ) B (Y, E_2, E_3, E_4 ) - g(E_1 ,Y)  B ( \nabla \scal, E_2, E_3, E_4 ) \\
& \ - d \scal( JE_1 ) B (JY, E_2, E_3, E_4 ) + g( JE_1 ,Y)  B (J \nabla \scal, E_2, E_3, E_4 ) \\
& \ - 2 d \scal(JY) B( JE_1, E_2, E_3, E_4 ) \\[2mm]
& \ + d \scal( E_2 ) B (E_1, Y, E_3, E_4 ) - g(E_2 ,Y)  B (E_1, \nabla \scal, E_3, E_4 ) \\
& \ -  d \scal( JE_2 ) B (E_1, JY, E_3, E_4 ) + g( JE_2 ,Y)  B (E_1, J \nabla \scal, E_3, E_4 ) \\
& \ - 2 d \scal(JY) B( E_1, JE_2, E_3, E_4 ) \\[2mm]
& \ + d \scal( E_3 ) B (E_1, E_2, Y, E_4 ) - g(E_3 ,Y)  B (E_1, E_2, \nabla \scal, E_4, ) \\
& \ - d \scal( JE_3 ) B (E_1, E_2, JY, E_4 ) + g( JE_3 ,Y)  B (E_1, E_2, J \nabla \scal, E_4 ) \\
& \ - 2 d \scal(JY) B( E_1, E_2, JE_3, E_4 ) \\[2mm]
& \ + d \scal( E_4 ) B (E_1, E_2, E_3, Y ) - g(E_4 ,Y)  B (E_1, E_2, E_3, \nabla \scal ) \\
& \ - d \scal( JE_4 ) B (E_1, E_2, E_3, JY ) + g( JE_4 ,Y)  B (E_1, E_2, E_3, J \nabla \scal ) \\
& \ - 2 d \scal(JY) B( E_1, E_2, E_3, JE_4 ).
\end{align*}
Note that 
\begin{align*}
& B( JE_1, E_2, E_3, E_4 ) + B( E_1, JE_2, E_3, E_4 ) 
+B( E_1, E_2, JE_3, E_4 ) +B( E_1, E_2, E_3, JE_4 ) \\
& = B( JE_1, E_2, E_3, E_4 ) - B( JE_1, E_2, E_3, E_4 ) 
+B( E_1, E_2, JE_3, E_4 ) - B( E_1, E_2, JE_3, E_4 ) =  0
\end{align*}
and hence the terms with factors of $2 d \scal(JY)$ cancel.

Therefore we obtain
\begin{align*}
0 = & \ d \scal( E_1 ) B (Y, E_2, E_3, E_4 ) - g(E_1 ,Y)  B ( \nabla \scal, E_2, E_3, E_4 ) \\
& \ + d \scal( E_2 ) B (E_1, Y, E_3, E_4 ) - g(E_2 ,Y)  B (E_1, \nabla \scal, E_3, E_4 ) \\
& \ + d \scal( E_3 ) B (E_1, E_2, Y, E_4 ) - g(E_3 ,Y)  B (E_1, E_2, \nabla \scal, E_4, ) \\
& \ + d \scal( E_4 ) B (E_1, E_2, E_3, Y ) - g(E_4 ,Y)  B (E_1, E_2, E_3, \nabla \scal ) \\[2mm]
& \ - d \scal( JE_1 ) B (JY, E_2, E_3, E_4 ) + g( JE_1 ,Y)  B (J \nabla \scal, E_2, E_3, E_4 ) \\
& \ -  d \scal( JE_2 ) B (E_1, JY, E_3, E_4 ) + g( JE_2 ,Y)  B (E_1, J \nabla \scal, E_3, E_4 ) \\
& \ - d \scal( JE_3 ) B (E_1, E_2, JY, E_4 ) + g( JE_3 ,Y)  B (E_1, E_2, J \nabla \scal, E_4 ) \\
& \ - d \scal( JE_4 ) B (E_1, E_2, E_3, JY ) + g( JE_4 ,Y)  B (E_1, E_2, E_3, J \nabla \scal ).
\end{align*}
In view of remark \ref{BochnerTotallyTraceFree}, contraction of $E_1$ with $Y$ yields
\begin{align*}
0 = & \ B ( \nabla \scal, E_2, E_3, E_4 ) - 2m  B ( \nabla \scal, E_2, E_3, E_4 ) \\
& \ - B (E_2, \nabla \scal, E_3, E_4 ) \\
& \ - B (E_3, E_2, \nabla \scal, E_4, ) \\
& \ - B (E_4, E_2, E_3, \nabla \scal ) \\[2mm]
& \ - B ( \nabla \scal, E_2, E_3, E_4 )  \\
& \ + B (JE_2, J \nabla \scal,  E_3, E_4 ) \\
& \ + B (JE_3, E_2, J \nabla \scal, E_4 ) \\
& \ + B (JE_4, E_2, E_3, J \nabla \scal) \\[2mm]
= & \ - 2m  B ( \nabla \scal, E_2, E_3, E_4 ) \\
& \ - B (E_3, E_4, E_2, \nabla \scal ) - B (E_4, E_3, JE_2, J \nabla \scal ) \\
& \ - B (E_2, E_3, E_4, \nabla \scal ) + B (E_3, JE_2, E_4, J \nabla \scal ) \\
& \ - B (E_4, E_2, E_3, \nabla \scal ) + B (JE_2, E_4, E_3, J \nabla \scal) \\[2mm]
= & \ - 2m  B ( \nabla \scal, E_2, E_3, E_4 ) - 2 B (E_4, E_3, JE_2, J \nabla \scal) \\
= & \ - 2(m+1)  B ( \nabla \scal, E_2, E_3, E_4 ),
\end{align*}
where we used the algebraic Bianchi identity in the penultimate step.

Therefore we conclude that
\begin{align*}
B( \nabla \scal, \cdot, \cdot, \cdot ) = B( J \nabla \scal, \cdot, \cdot, \cdot ) = 0
\end{align*}
since $B( J \cdot , J \cdot, \cdot, \cdot ) = B( \cdot ,  \cdot, \cdot, \cdot ) = 0.$

Inserting this back into the above equation we find
\begin{align*}
0 = & \ d \scal( E_1 ) B (Y, E_2, E_3, E_4 ) - d \scal( JE_1 ) B (JY, E_2, E_3, E_4 ) \\
& \ + d \scal( E_2 ) B (E_1, Y, E_3, E_4 ) -  d \scal( JE_2 ) B (E_1, JY, E_3, E_4 ) \\
& \ + d \scal( E_3 ) B (E_1, E_2, Y, E_4 ) - d \scal( JE_3 ) B (E_1, E_2, JY, E_4 ) \\
& \ + d \scal( E_4 ) B (E_1, E_2, E_3, Y )  - d \scal( JE_4 ) B (E_1, E_2, E_3, JY ).
\end{align*}

Finally, set $E_1 = \nabla \scal$ and note that $d \scal (J \nabla \scal ) = g( J \nabla \scal, \nabla \scal )=0$ to conclude that
\begin{align*}
| \nabla \scal |^2 \cdot B = 0
\end{align*}
as required.
\end{proof}

\textit{Proof} of Theorem \ref{MainTheoremBochner}. It follows from corollary \ref{BochnerTechniqueForBochnerTensor} that the Bochner tensor is parallel. Proposition \ref{ParallelBochnerImpliesConstantScalorBZero} shows that hence the scalar curvature is constant or the Bochner tensor vanishes. 

According to a theorem of Kim \cite{KimHarmonicBochner}, a K\"ahler manifold with divergence free Bochner tensor and constant scalar curvature has parallel Ricci tensor. Therefore, in this case, the curvature tensor is in fact parallel and $(M,g)$ is locally symmetric. 

On the other hand, if $(M,g)$ is Bochner flat, then it is locally symmetric due to Bryant's classification of compact Bochner flat K\"ahler manifolds in \cite[Corollary 4.17]{BryantBochnerKaehlerMetrics}.  $\hfill \Box$

\section{Quaternion K\"ahler manifolds}
\label{SectionQKmanifolds}

A Riemannian manifold with holonomy contained in $Sp(m) \cdot Sp(1)$ is called quaternion K\"ahler manifold. Since $Sp(1) \cdot Sp(1) = SO(4),$ we will restrict to the case $m \geq 2.$ 

Locally there exist almost complex structures $I, J, K$ such that $IJ = - JI = K.$ For a local orthonormal frame $\lbrace e_i, Ie_i, Je_i, Ke_i \rbrace_{i=1, \ldots, m}$ consider
\begin{align*}
\omega_I & = \sum_{i=1}^m e_i \wedge I e_i + Je_i \wedge K e_i, \\
\omega_J & = \sum_{i=1}^m e_i \wedge J e_i + Ke_i \wedge I e_i, \\
\omega_K & = \sum_{i=1}^m e_i \wedge K e_i + Ie_i \wedge J e_i. 
\end{align*}
It is straightforward to check that
\begin{align*}
g(IX,Y) = g(X \wedge Y, \omega_I ), \ g(JX,Y) = g(X \wedge Y, \omega_J ), \ g(KX,Y) = g(X \wedge Y, \omega_K ).
\end{align*}
The curvature operator of quaternionic projective space is given by
\begin{align*}
\mathfrak{R}_{\mathbb{HP}^m}(X \wedge Y ) = & \  X \wedge Y + IX \wedge IY + JX \wedge JY + KX \wedge KY \\
& \ + 2 g(X \wedge Y, \omega_I ) \ \omega_I + 2 g(X \wedge Y, \omega_J ) \ \omega_J + 2 g(X \wedge Y, \omega_K ) \ \omega_K.
\end{align*}

\begin{remark} \normalfont
\label{EigenspaceCurvatureOperatorHPm}
In this normalization of the metric, the curvature operator has eigenvalues $4m$ and $4.$ The eigenspace for the eigenvalue $4m$ is isomorphic to $\mathfrak{sp}(1)$ and spanned by $\omega_I, \omega_J, \omega_K.$ The eigenspace for the eigenvalue $4$ is isomorphic to $\mathfrak{sp}(m)$ and spanned by
\begin{align*}
\mathcal{W}_{ij} & =  \frac{1}{2} \left( e_i \wedge e_j + Ie_i \wedge Ie_j + Je_i \wedge Je_j + Ke_i \wedge Ke_j \right) \ \text{for} \ 1 \leq i < j \leq m, \\[2mm]
\mathcal{I}_{ij} & = \frac{1}{2} \left( e_i \wedge Ie_j + e_j \wedge Ie_i - Je_i \wedge Ke_j - Je_j \wedge Ke_i \right), \\
\mathcal{J}_{ij} & = \frac{1}{2} \left( e_i \wedge Je_j + e_j \wedge Je_i - Ke_i \wedge Ie_j - Ke_j \wedge Ie_i \right), \\
\mathcal{K}_{ij} & = \frac{1}{2} \left( e_i \wedge Ke_j + e_j \wedge Ke_i - Ie_i \wedge Je_j - Ie_j \wedge Je_i \right) \ \text{for} \ 1 \leq i < j \leq m, \\[2mm]
\mathcal{I}_i & = \frac{1}{\sqrt{2}} \left( e_i \wedge I e_i - Je_i \wedge K e_i \right), \\
\mathcal{J}_i & = \frac{1}{\sqrt{2}} \left(  e_i \wedge J e_i - Ke_i \wedge I e_i \right), \\
\mathcal{K}_i & =  \frac{1}{\sqrt{2}} \left(  e_i \wedge K e_i - Ie_i \wedge J e_i \right) \ \text{for} \ i=1, \ldots, m.
\end{align*}
Recall that $\dim \mathfrak{sp}(m) = m (2m+1).$ In particular,
\begin{align*}
| R_{\mathbb{HP}^m} |^2 = 16m (5m+1) \ \text{ and } \ 
\scal( R_{\mathbb{HP}^m} ) = 16 m (m+2).
\end{align*} 
\end{remark}

The curvature operator $R \in \operatorname{Sym}_B^2(TM)$ of a quaternion K\"ahler manifold satisfies
\begin{align*}
R = \frac{\scal}{16 m (m+2)} R_{\mathbb{HP}^m} + R_0, 
\end{align*} 
where $R_0$ is the hyper-K\"ahler component. Recall that hyper-K\"ahler manifolds have holonomy contained in $Sp(m)$ and are necessarily Ricci flat. 

Due to a result of Alekseevski \cite{AlekseevskiiRiemannianSpacesExceptionalHol}, see also \cite{SalamonQKManifolds}, this is indeed the decomposition of the curvature tensor of a quaternion K\"ahler manifold according to the decomposition of the space of quaternion K\"ahler curvature tensors into orthogonal, $Sp(m) \cdot Sp(1)$-invariant, irreducible subspaces.

The key ingredient in the proof of Theorem \ref{MainQKTheorem} is the computation of $| R^{\mathfrak{sp}(m) \oplus \mathfrak{sp}(1)}|^2$ for any quaternion K\"ahler curvature tensor $R \in \operatorname{Sym}_B^2( \mathfrak{sp}(m) \oplus \mathfrak{sp}(1) )$ in corollary \ref{ReducedHatQKCurvatureTensors}. The curvature tensor of quaternionic projective space satisfies $| R_{\mathbb{HP}^m}^{\mathfrak{sp}(m) \oplus \mathfrak{sp}(1)}|^2 = 0$ due to the following observation.

\begin{proposition}
Let $(M,g)$ be an isotropy irreducible symmetric space with holonomy algebra $\mathfrak{hol}.$

The curvature tensor $R \in \operatorname{Sym}^2( \mathfrak{hol})$ satisfies
\begin{align*}
| R^{\mathfrak{hol}} | = 0.
\end{align*}

In particular, 
\begin{align*}
LR = 0
\end{align*}
for all $L \in \mathfrak{hol}$.
\label{HatIsotropyIrreducibleSpaces}
\end{proposition}
\begin{proof}
Let $\left\lbrace \Xi_{\alpha} \right\rbrace$ be an orthonormal eigenbasis for the curvature operator 
$\mathfrak{R}_{| \mathfrak{hol}} \colon \mathfrak{hol} \to \mathfrak{hol}$ and let $\lbrace \lambda_{\alpha} \rbrace$ denote the corresponding eigenvalues. According to \cite[Example 1.2]{PetersenWinkVanishingHodgeNumbers} we have 
\begin{align*}
| R^{\mathfrak{hol}} |^2 
= \sum_{\alpha} | \Xi_{\alpha} R |^2 
= \sum_{\gamma} \sum_{\alpha, \beta} ( \lambda_{\alpha} - \lambda_{\beta} )^2 g( ( \Xi_{\alpha} ) \Xi_{\beta}, \Xi_{\gamma} )^2.
\end{align*}

Hence we may assume that $\Xi_{\alpha}$, $\Xi_{\beta}$ correspond to different eigenvalues $\lambda_{\alpha} \neq \lambda_{\beta}.$ However, recall that we can identify the isotropy representation with the holonomy representation and the adjoint representation, respectively. Thus, by assumption, we have $( \Xi_{\alpha} ) \Xi_{\beta} = [  \Xi_{\alpha}, \Xi_{\beta} ] =  0$ whenever $\lambda_{\alpha} \neq \lambda_{\beta}.$
\end{proof}

The computation of $|R^{\mathfrak{sp}(m) \oplus \mathfrak{sp}(1)}|^2$ in terms of the hyper-K\"ahler component $|R_0|^2$ in corollary   \ref{ReducedHatQKCurvatureTensors} is based on the computation of $|R^{\mathfrak{sp}(m) \oplus \mathfrak{sp}(1)}|^2$ for the Wolf spaces $\frac{SO(m+4)}{S(O(m) \times O(4))}$. 

\begin{example} \normalfont
Let
\begin{align*}
G_{\mathbb{R}}(p,q) = \frac{O(p+q)}{O(p) \times O(q)}
\end{align*}
denote the Grassmannian of $p$-planes in $\R^{p+q}$. Under the identification
\begin{align*}
\mathfrak{t}_x = \left\lbrace 
\begin{pmatrix}
0 & - X^T \\
X & 0
\end{pmatrix} \ \middle| \ X \in \R^{q \times p} \right\rbrace \cong \R^{q \times p}
\end{align*}
of the tangent space with $\R^{q \times p}$ the metric is given by $g(X,Y)=\tr (X^TY)$ and the curvature tensor by
\begin{align*}
R(X,Y,Z,W) = \tr \left( -Z^T Y X^T W - X^T Y Z^T W + Z^T X Y^T W + Y^T X Z^T W \right),
\end{align*}
cf. \cite[Example B.42]{BallmannKaehlerGeometry}. In particular, if $E_{ij}$ denotes the standard orthonormal basis of $\R^{q \times p}$, then the eigenspaces of the curvature operator are given by
\begin{align*}
\operatorname{Eig}(p)  = & \  \mathfrak{so}(q) = \operatorname{span} \left\lbrace \sum_{i=1}^p E_{ki} \wedge E_{li} \ \middle| \ 1 \leq k < l \leq q \right\rbrace, \\
\operatorname{Eig}(q) = & \ \mathfrak{so}(p) = \operatorname{span} \left\lbrace \sum_{i=1}^q E_{ik} \wedge E_{il} \ \middle| \ 1 \leq k < l \leq p \right\rbrace, \\
\operatorname{ker}(\mathfrak{R}) = & \ V_p \oplus V_q \oplus \operatorname{span} \left\lbrace E_{ab} \wedge E_{cd} \ \middle| \ a \neq c, \ b \neq d \right\rbrace,
\end{align*}
where 
\begin{align*}
V_p = & \ \operatorname{Eig}(p)^{\perp} \subset \operatorname{span}  \left\lbrace E_{ki} \wedge E_{li} \ \middle| \begin{tabular}{l}
$1 \leq i \leq p,$ \\
$1 \leq k < l \leq q$
\end{tabular} \right\rbrace, \\
%= & \ \operatorname{span} \left\lbrace \frac{1}{\sqrt{i+i^2}}\left(- i \cdot E_{k,i+1} \wedge E_{l,i+1} + \sum_{j=1}^i E_{kj} \wedge E_{lj} \right) \ \middle| 
%\begin{tabular}{l}
%$1 \leq i \leq p-1,$ \\
%$1 \leq k < l \leq q$
%\end{tabular}     \right\rbrace \\
V_q = & \ \operatorname{Eig}(q)^{\perp} \subset \operatorname{span}  \left\lbrace E_{ik} \wedge E_{il} \ \middle| \begin{tabular}{l}
$1 \leq i \leq q,$ \\
$1 \leq k < l \leq p$
\end{tabular} \right\rbrace
% \\
%= & \ \operatorname{span} \left\lbrace \frac{1}{\sqrt{i+i^2}}\left(- i \cdot E_{i+1,k} \wedge E_{i+1,l} + \sum_{j=1}^i E_{jk} \wedge E_{lj} \right) \ \middle| 
%\begin{tabular}{l}
%$1 \leq i \leq q-1,$ \\
%$1 \leq k < l \leq p$
%\end{tabular} \right\rbrace
.
\end{align*}
\label{CurvatureOperatorGrassmannian}
\end{example}

\begin{proposition}
\label{ReducedHatWolfSpace}
The curvature tensor $R_W \in \operatorname{Sym}_B^2(TM)$ of the Wolf space 
\begin{align*}
\frac{SO(m+4)}{S(O(m) \times O(4))}
\end{align*}
satisfies 
\begin{align*}
| R_W^{\mathfrak{sp}(m) \oplus \mathfrak{sp}(1)} |^2 & = 36m^2(m-1), \\
|R_W|^2 & = 2m(7m-4), \\
\scal(R_W) & = 4 m (m+2).
\end{align*}
\end{proposition}
\begin{proof}
Example \ref{CurvatureOperatorGrassmannian} exhibits the geometry of the Grassmannians. To emphasize the quaternion K\"ahler structure we use the identification
\begin{align*}
X= (x_{ij}) \mapsto \sum_{j=1}^m \left( x_{1j} e_j + x_{2j} Ie_j + x_{3j} Je_j + x_{4j} Ke_j \right).
\end{align*}

It is straightforward to describe the eigenspaces $\operatorname{Eig}(m)$ and $\operatorname{Eig}(4)$ in terms of the quaternion K\"ahler geometry. Moreover, using the eigenbasis of the curvature operator of $\mathbb{HP}^m$ in remark \ref{EigenspaceCurvatureOperatorHPm} as a basis for $\mathfrak{sp}(m) \oplus \mathfrak{sp}(1)$, it is easy to find an orthonormal eigenbasis of the curvature operator 
\begin{align*}
\mathfrak{R}_W = \mathfrak{R}_{|\mathfrak{sp}(m) \oplus \mathfrak{sp}(1)} \colon \mathfrak{sp}(m) \oplus \mathfrak{sp}(1) \to \mathfrak{sp}(m) \oplus \mathfrak{sp}(1).
\end{align*}

Specifically, let 
\begin{align*}
\omega_I^+ & = \frac{1}{\sqrt{2m}}  \sum_{i=1}^m \left( e_i \wedge I e_i + Je_i \wedge Ke_i \right), \
\omega_I^- = \frac{1}{\sqrt{2m}}  \sum_{i=1}^m \left( e_i \wedge I e_i - Je_i \wedge Ke_i \right), \\
\omega_J^+ & = \frac{1}{\sqrt{2m}}  \sum_{i=1}^m \left( e_i \wedge J e_i + Ke_i \wedge Ie_i \right),  \
\omega_J^- = \frac{1}{\sqrt{2m}}  \sum_{i=1}^m \left( e_i \wedge J e_i - Ke_i \wedge Ie_i \right), \\
\omega_K^+ & = \frac{1}{\sqrt{2m}}  \sum_{i=1}^m \left( e_i \wedge K e_i + Ie_i \wedge Je_i \right), \
\omega_K^- = \frac{1}{\sqrt{2m}}  \sum_{i=1}^m \left( e_i \wedge K e_i - Ie_i \wedge Je_i \right)
\end{align*}
and for $\mathcal{L} = \mathcal{I}, \mathcal{J}, \mathcal{K}$ define
\begin{align*}
\widetilde{\mathcal{L}}_i & = \frac{1}{\sqrt{i^2 + i}} \left( -i \cdot \mathcal{L}_{i+1} + \sum_{j=1}^i \mathcal{L}_j \right).
\end{align*}
Recall that $\mathcal{I}_i, \mathcal{J}_i, \mathcal{K}_i$ and $\mathcal{W}_{ij}$ are defined in remark \ref{EigenspaceCurvatureOperatorHPm}.

%\begin{align*}
%\widetilde{\mathcal{I}}_i & = \frac{1}{\sqrt{i^2 + i}} \left( -i \cdot \mathcal{I}_{i+1} + \sum_{j=1}^i \mathcal{I}_j \right), \\
%\widetilde{\mathcal{J}}_i & = \frac{1}{\sqrt{i^2 + i}} \left( -i \cdot \mathcal{J}_{i+1} + \sum_{j=1}^i \mathcal{J}_j \right), \\
%\widetilde{\mathcal{K}}_i & = \frac{1}{\sqrt{i^2 + i}} \left( -i \cdot \mathcal{K}_{i+1} + \sum_{j=1}^i \mathcal{K}_j \right)
%\end{align*}
%where $\mathcal{I}_i, \mathcal{J}_i, \mathcal{K}_i$ are defined in remark \ref{EigenspaceCurvatureOperatorHPm}.

It follows that the eigenspaces for $\mathfrak{R}_W \colon \mathfrak{sp}(m) \oplus \mathfrak{sp}(1) \to \mathfrak{sp}(m) \oplus \mathfrak{sp}(1)$ 
are given by
\begin{align*}
\operatorname{Eig}(m) = & \ \mathfrak{sp}(1) \oplus \mathfrak{sp}(1) = \operatorname{span} \left\lbrace \omega_I^+, \omega_J^+, \omega_K^+ \right\rbrace \oplus \operatorname{span} \left\lbrace \omega_I^-, \omega_J^-, \omega_K^- \right\rbrace, \\
\operatorname{Eig}(4) = & \ \mathfrak{so}(m) = \operatorname{span} \left\lbrace \mathcal{W}_{ij} \ \vert \ \text{for} \ 1 \leq i < j \leq m \right\rbrace, \\
\operatorname{ker} \left( \mathfrak{R}_W \right) = & \ \operatorname{span} \left\lbrace \mathcal{I}_{ij}, \mathcal{J}_{ij}, \mathcal{K}_{ij} \ \vert \ \text{for} \ 1 \leq 1 < j \leq m \right\rbrace   \\
& \ \oplus \operatorname{span} \left\lbrace \widetilde{\mathcal{I}}_i, \widetilde{\mathcal{J}}_i, \widetilde{\mathcal{K}}_i \ \vert \ \text{for} \ i=1, \ldots, m-1 \right\rbrace.
\end{align*}
In particular, 
\begin{align*}
|R_W|^2 & = | \mathfrak{R}_W|^2 = 2m(7m-4), \\
\scal(R_W) & = 2 \tr (\mathfrak{R}_W) = 4 m (m+2).
\end{align*}

In the following, we will consider the orthonormal eigenbasis
\begin{align*}
\mathcal{B}_0= \left\lbrace  \omega_I^+, \omega_J^+, \omega_K^+, \omega_I^-, \omega_J^-, \omega_K^-, \mathcal{W}_{ij},  \mathcal{I}_{ij},  \mathcal{J}_{ij},  \mathcal{K}_{ij},  \widetilde{\mathcal{I}}_k, \widetilde{\mathcal{J}}_k, \widetilde{\mathcal{K}}_k  \right\rbrace 
\end{align*}
for $\mathfrak{R}_W.$

According to \cite[Example 1.2]{PetersenWinkVanishingHodgeNumbers}, given an orthonormal eigenbasis $\left\lbrace \Xi_{\alpha} \right\rbrace $ for $\mathfrak{R}_W,$
\begin{align*}
| R_W^{\mathfrak{sp}(m) \oplus \mathfrak{sp}(1)} |^2 = 2 \sum_{\gamma} \sum_{\alpha < \beta} ( \lambda_{\alpha} - \lambda_{\beta} )^2 g( ( \Xi_{\alpha} ) \Xi_{\beta}, \Xi_{\gamma} )^2.
\end{align*}

We will compute the overall sum by separately evaluating
\begin{align*}
 2 \ \sum_{\Xi_{\gamma} \in \mathcal{B}_0}  \ \ \sum_{\Xi_{\alpha} \in \mathcal{B}_i, \ \Xi_{\beta} \in \mathcal{B}_j} \ \ ( \lambda_{\alpha} - \lambda_{\beta} )^2 \ g( ( \Xi_{\alpha} ) \Xi_{\beta}, \Xi_{\gamma} )^2
\end{align*}
for orthonormal bases $\mathcal{B}_i$ for suitable subspaces of $\operatorname{Eig}(m)$,  $\operatorname{Eig}(4)$,  $\operatorname{ker}(\mathfrak{R}_W)$.

Recall from the proof of proposition \ref{HatIsotropyIrreducibleSpaces} that the Lie algebra action of $\operatorname{Eig}(m)$ on  $\operatorname{Eig}(4)$ is trivial. 

Next we consider the action of $\operatorname{Eig}(m)$ on $\operatorname{ker}(\mathfrak{R}_W).$ Clearly,
\begin{align*}
\mathcal{B}_1= \left\lbrace  \omega_I^+, \omega_J^+, \omega_K^+, \omega_I^-, \omega_J^-, \omega_K^- \right\rbrace 
\end{align*}
is an orthonormal eigenbasis for $\operatorname{Eig}(m)$.

Remark \ref{EigenspaceCurvatureOperatorHPm} shows that $\mathfrak{sp}(1) = \operatorname{span} \left\lbrace \omega_I^+, \omega_J^+, \omega_K^+ \right\rbrace$ is part of the isotropy of $\mathbb{HP}^m.$ Thus it acts trivially on $\operatorname{ker} \left( \mathfrak{R}_W \right) \subset \mathfrak{sp}(m)$. 
%The Lie algebra $\mathfrak{sp}(1) = \operatorname{span} \left\lbrace \omega_I^+, \omega_J^+, \omega_K^+ \right\rbrace$ acts trivially on $\operatorname{ker} \left( \mathfrak{R}_W \right)$. Indeed $\omega_I^+, \omega_J^+, \omega_K^+$ act trivially on $\mathcal{I}_{ij}, \mathcal{J}_{ij}, \mathcal{K}_{ij}$ as well as $\mathcal{I}_i, \mathcal{J}_i, \mathcal{K}_i,$ hence also $\widetilde{\mathcal{I}}_i, \widetilde{\mathcal{J}}_i, \widetilde{\mathcal{K}}_i.$
On the other hand, for the Lie algebra $\mathfrak{sp}(1) = \operatorname{span} \left\lbrace \omega_I^-, \omega_J^-, \omega_K^- \right\rbrace,$ we obtain
\begin{center}
\begin{tabular}{c|c|c|c}
 & $\mathcal{I}_{ij}$ & $\mathcal{J}_{ij}$ & $\mathcal{K}_{ij}$ \\[2mm] \hline & & & \\[-2mm]
$ ( \omega_I^- ) \cdot$ & $0$ & $- \sqrt{\frac{2}{m}} \mathcal{K}_{ij}$ & $\sqrt{\frac{2}{m}} \mathcal{J}_{ij}$  \\[2mm] \hline & & & \\[-2mm]
$ ( \omega_J^- ) \cdot$ & $\sqrt{\frac{2}{m}} \mathcal{K}_{ij}$ & $0$ & $- \sqrt{\frac{2}{m}} \mathcal{I}_{ij}$  \\[2mm] \hline & & & \\[-2mm]
$ ( \omega_K^- ) \cdot$ & $- \sqrt{\frac{2}{m}} \mathcal{J}_{ij}$ & $ \sqrt{\frac{2}{m}} \mathcal{I}_{ij}$ & $0$
\end{tabular}
\hspace{8mm}
\begin{tabular}{c|c|c|c}
 & $\widetilde{\mathcal{I}}_i$ & $\widetilde{\mathcal{J}}_i$ & $\widetilde{\mathcal{K}}_i$ \\[2mm] \hline & & & \\[-2mm]
$ ( \omega_I^- ) \cdot$ & $0$ & $- \sqrt{\frac{2}{m}} \widetilde{\mathcal{K}}_i$ & $\sqrt{\frac{2}{m}} \widetilde{\mathcal{J}}_i$  \\[2mm] \hline & & & \\[-2mm]
$ ( \omega_J^- ) \cdot$ & $\sqrt{\frac{2}{m}} \widetilde{\mathcal{K}}_i$ & $0$ & $- \sqrt{\frac{2}{m}} \widetilde{\mathcal{I}}_i$  \\[2mm] \hline & & & \\[-2mm]
$ ( \omega_K^- ) \cdot$ & $- \sqrt{\frac{2}{m}} \widetilde{\mathcal{J}}_i$ & $ \sqrt{\frac{2}{m}} \widetilde{\mathcal{I}}_i$ & $0$
\end{tabular}
\end{center}
\begin{comment}
\begin{center}
\begin{tabular}{c|c|c|c}
 & $\mathcal{I}_i$ & $\mathcal{J}_i$ & $\mathcal{K}_i$ \\[2mm] \hline & & & \\[-2mm]
$ ( \omega_I^- ) \cdot$ & $0$ & $- \sqrt{\frac{2}{m}} \mathcal{K}_i$ & $\sqrt{\frac{2}{m}} \mathcal{J}_i$  \\[2mm] \hline & & & \\[-2mm]
$ ( \omega_J^- ) \cdot$ & $\sqrt{\frac{2}{m}} \mathcal{K}_i$ & $0$ & $- \sqrt{\frac{2}{m}} \mathcal{I}_i$  \\[2mm] \hline & & & \\[-2mm]
$ ( \omega_K^- ) \cdot$ & $- \sqrt{\frac{2}{m}} \mathcal{J}_i$ & $ \sqrt{\frac{2}{m}} \mathcal{I}_i$ & $0$
\end{tabular}
\end{center}
\end{comment}
In fact, a similar diagram is valid for $\mathcal{I}_i,$ $\mathcal{J}_i,$ $\mathcal{K}_i$. %In fact, a similar diagram is valid when replacing $\widetilde{\mathcal{I}}_i,$ $\widetilde{\mathcal{J}}_i,$ $\widetilde{\mathcal{K}}_i$ with $\mathcal{I}_i,$ $\mathcal{J}_i,$ $\mathcal{K}_i$. The same diagrams are valid when replacing $\mathcal{I}_i,$ $\mathcal{J}_i,$ $\mathcal{K}_i$ with $\mathcal{I}_{ij},$ $\mathcal{J}_{ij},$ $\mathcal{K}_{ij}$ or $\widetilde{\mathcal{I}}_i,$ $\widetilde{\mathcal{J}}_i,$ $\widetilde{\mathcal{K}}_i$, respectively.
If
\begin{align*}
\mathcal{B}_2 = \left\lbrace \mathcal{I}_{ij}, \mathcal{J}_{ij}, \mathcal{K}_{ij} \ \vert \ \text{for} \ 1 \leq i < j \leq m \right\rbrace \ \text{and} \
\mathcal{B}_3 = \left\lbrace \widetilde{\mathcal{I}}_i, \widetilde{\mathcal{J}}_i, \widetilde{\mathcal{K}}_i \ \vert \ \text{for} \ i=1, \ldots, m-1 \right\rbrace,
\end{align*}
then $\mathcal{B}_2 \cup \mathcal{B}_3$ is an orthonormal eigenbasis for $\operatorname{ker}(\mathfrak{R}_W).$ Observe that
\begin{align*}
 2 \ \sum_{\Xi_{\gamma}  \in \mathcal{B}_0} & \ \ \sum_{\Xi_{\alpha} \in \mathcal{B}_1, \ \Xi_{\beta} \in \mathcal{B}_2} \ \ ( \lambda_{\alpha} - \lambda_{\beta} )^2 \ g( ( \Xi_{\alpha} ) \Xi_{\beta}, \Xi_{\gamma} )^2 \\
 = & \ 6  \sum_{\Xi_{\gamma}  \in \mathcal{B}_0} \sum_{\Xi_{\beta}  \in \mathcal{B}_2} m^2 \cdot g \left( ( \omega_I^-)\Xi_{\beta}, \Xi_{\gamma} \right)^2 \\
 = & \ 12m^2 \sum_{1\leq i < j \leq m} g \left( ( \omega_I^-)\mathcal{J}_{ij}, \mathcal{K}_{ij} \right)^2 \\
 = & \ 12m^2 \sum_{1\leq i < j \leq m} \left( - \sqrt{\frac{2}{m}} \right)^2 =12m^2(m-1)
\end{align*}
and similarly
\begin{align*}
 2 \ \sum_{\Xi_{\gamma}  \in \mathcal{B}_0} & \ \ \sum_{\Xi_{\alpha} \in \mathcal{B}_1, \ \Xi_{\beta} \in \mathcal{B}_3} \ \ ( \lambda_{\alpha} - \lambda_{\beta} )^2 \ g( ( \Xi_{\alpha} ) \Xi_{\beta}, \Xi_{\gamma} )^2 \\
 = & \ 6  \sum_{\Xi_{\gamma}  \in \mathcal{B}_0} \sum_{\Xi_{\beta}  \in \mathcal{B}_3} m^2 \cdot g \left( ( \omega_I^-)\Xi_{\beta}, \Xi_{\gamma} \right)^2 \\
 = & \ 12m^2 \sum_{i=1}^{m-1} g \left( ( \omega_I^-)\widetilde{\mathcal{J}}_{i}, \widetilde{\mathcal{K}}_{i} \right)^2 \\
 = & \ 12m^2 \sum_{i=1}^{m-1} \left( - \sqrt{ \frac{2}{m}} \right)^2 =24m(m-1).
\end{align*}

Finally consider the action of $\operatorname{Eig}(4)$ on $\operatorname{ker}(\mathfrak{R}_W).$ Let 
\begin{align*}
\mathcal{B}_4 = \left\lbrace \mathcal{W}_{ij} \ \vert \ \text{for} \ 1 \leq i < j \leq m \right\rbrace 
\end{align*}
denote an orthonormal eigenbasis of $\operatorname{Eig}(4)$.

Firstly, we compute the action of $\mathcal{B}_4$ on $\mathcal{B}_2.$ For two sets $A, B$ let $A \Delta B = ( A \cup B ) \setminus ( A \cap B )$ denote the symmetric difference. It is straightforward to check that for $i<j$, $k<l$ and $\mathcal{L}= \mathcal{I}, \mathcal{J}, \mathcal{K}$ 
\begin{align*}
( \mathcal{W}_{ij} ) \mathcal{L}_{kl} = 
\begin{cases}
0 & \text{for} \ \lbrace i,j \rbrace \cap \lbrace k,l \rbrace = \emptyset, \\
\pm \frac{1}{2} \mathcal{L}_{\alpha \beta} & \text{for} \ | \lbrace i,j \rbrace \cap \lbrace k,l \rbrace | = 1, \ \text{where} \ \lbrace \alpha,  \beta \rbrace = \lbrace i,j \rbrace \Delta \lbrace k,l \rbrace, \\
\frac{1}{\sqrt{2}} \left( - \mathcal{L}_i + \mathcal{L}_j \right) & \text{for} \ \lbrace i,j \rbrace = \lbrace k,l \rbrace,
\end{cases}
\end{align*}
where $\mathcal{L}_k$ is defined in remark \ref{EigenspaceCurvatureOperatorHPm} for $\mathcal{L} = \mathcal{I}, \mathcal{J}, \mathcal{K}.$

Furthermore, $( \mathcal{W}_{ij} ) \mathcal{L}_{ij} = \frac{1}{\sqrt{2}} \left( - \mathcal{L}_i + \mathcal{L}_j \right)$ is orthogonal to all basis elements in $\mathcal{B}_0$ except possibly $\widetilde{\mathcal{L}}_k$. Note that 
\begin{align*}
g \left(  \mathcal{L}_{i}, \widetilde{\mathcal{L}}_{k} \right) = 
\begin{cases}
0 & \ \text{for} \ k+1 < i, \\
- \frac{k}{\sqrt{k^2+k}} & \ \text{for} \ k+1=i, \\
\frac{1}{\sqrt{k^2+k}} & \ \text{for} \ i < k+1.
\end{cases}
\end{align*}
Thus the only non-zero inner products of $( \mathcal{W}_{ij} ) \mathcal{L}_{ij}$ with elements in $\mathcal{B}_0$ are given by
\begin{align*}
g \left( ( \mathcal{W}_{ij} ) \mathcal{L}_{ij}, \widetilde{\mathcal{L}}_k \right) = 
\begin{cases}
0 & \ \text{for} \ k+1 < i < j, \\
\frac{k}{\sqrt{2} \sqrt{k^2+k}} & \ \text{for} \ k+1 = i < j, \\
- \frac{1}{\sqrt{2} \sqrt{k^2+k}} & \ \text{for} \ i < k+1 < j, \\
- \frac{k+1}{\sqrt{2} \sqrt{k^2+k}} & \ \text{for} \ i < k+1 = j, \\
0 & \ \text{for} \ i < j < k+1
\end{cases}
\end{align*}
for $\mathcal{L} = \mathcal{I}, \mathcal{J}, \mathcal{K}$.

Note that there are $m(m-1)(m-3)$ many choices of $1 \leq i<j \leq m$ and $1 \leq k<l \leq m$ such that $|\lbrace i,j \rbrace \cap \lbrace k,l \rbrace|=1.$ It follows that
\begin{align*}
2 \ \sum_{\Xi_{\gamma}  \in \mathcal{B}_0} & \ \ \sum_{\Xi_{\alpha} \in \mathcal{B}_4, \ \Xi_{\beta} \in \mathcal{B}_2} \ \ ( \lambda_{\alpha} - \lambda_{\beta} )^2 \ g( ( \Xi_{\alpha} ) \Xi_{\beta}, \Xi_{\gamma} )^2 \\
= & \ 6 \sum_{\Xi_{\gamma}  \in \mathcal{B}_0} \sum_{1 \leq i < j \leq m} \sum_{1 \leq k < l \leq m} \ 4^2 \cdot g \left( (\mathcal{W}_{ij}) \mathcal{I}_{kl}, \Xi_{\gamma} \right)^2 \\
= & \ 2^5 \cdot 3 \sum_{\Xi_{\gamma}  \in \mathcal{B}_0} \  \sum_{|\lbrace i,j \rbrace \cap \lbrace k,l \rbrace|=1} \ g \left( (\mathcal{W}_{ij}) \mathcal{I}_{kl}, \Xi_{\gamma} \right)^2 
+ 2^5 \cdot 3 \sum_{\Xi_{\gamma}  \in \mathcal{B}_0} \sum_{1 \leq i < j \leq m}   g \left( (\mathcal{W}_{ij}) \mathcal{I}_{ij}, \Xi_{\gamma} \right)^2 \\
= & \ 2^5 \cdot 3 \cdot m(m-1)(m-3) \cdot \left( \frac{1}{2} \right)^2 
+ 2^5 \cdot 3 \sum_{k=1}^{m-1} \sum_{1 \leq i < j \leq m}   g \left( (\mathcal{W}_{ij}) \mathcal{I}_{ij}, \widetilde{\mathcal{I}}_k \right)^2 \\
= & \ 24m(m-1)(m-3) \\
& \ + 2^5 \cdot 3 \sum_{k =1}^{m-1}  \left( \pm \frac{1}{ \sqrt{2} \sqrt{k^2+k}} \right)^2 \left\lbrace \sum_{k+1=i< j \leq m} k^2 \ + \sum_{1 \leq i < k+1 < j \leq m} 1 \ + \sum_{1 \leq i < k+1=j} (k+1)^2 \right\rbrace \\
= & \ 24m(m-1)(m-3)  + 48 \sum_{k =1}^{m-1}  \frac{1}{k^2+k} \left\lbrace k^2(m-(k+1))+k(m-k-1)+k(k+1)^2 \right\rbrace \\
%= & \ 24m(m-1)(m-3)  + 48 m \sum_{k =1}^{m-1}  \frac{k^2+k}{k^2+k} \\
= & \ 24m(m-1)(m-3)  + 48 m (m-1) = 24m(m-1)^2.
\end{align*}

Secondly, we compute the action of $\mathcal{B}_4$ on $\mathcal{B}_3.$ It is straightforward to check that $(\mathcal{W}_{ij}) \mathcal{L}_i = \frac{1}{\sqrt{2}} \mathcal{L}_{ij}$ and hence 
\begin{align*}
\left( \mathcal{W}_{ij} \right)\widetilde{\mathcal{L}}_k = \begin{cases}
\hspace{13mm} 0 & \ \text{for} \ k+1 < i < j, \\
- \frac{k}{\sqrt{2} \sqrt{k^2+k}} \ \mathcal{L}_{ij} & \ \text{for} \ k+1 = i < j, \\
\hspace{3mm} \frac{1}{\sqrt{2} \sqrt{k^2 + k}} \ \mathcal{L}_{ij} & \ \text{for} \ i < k+1 < j, \\
\hspace{3mm} \frac{k+1}{\sqrt{2} \sqrt{k^2 + k}} \ \mathcal{L}_{ij} & \ \text{for} \ i < k+1 = j, \\
\hspace{13mm} 0 & \ \text{for} \ i < j < k+1
\end{cases}
\end{align*}
for $\mathcal{L}= \mathcal{I}, \mathcal{J}, \mathcal{K}.$ Therefore, 
\begin{align*}
 2 \ \sum_{\Xi_{\gamma}  \in \mathcal{B}_0} & \ \ \sum_{\Xi_{\alpha} \in \mathcal{B}_4, \ \Xi_{\beta} \in \mathcal{B}_3} \ \ ( \lambda_{\alpha} - \lambda_{\beta} )^2 \ g( ( \Xi_{\alpha} ) \Xi_{\beta}, \Xi_{\gamma} )^2 \\
= & \ 6  \sum_{\Xi_{\gamma}  \in \mathcal{B}_0} \sum_{1\leq i < j \leq m} \sum_{k =1}^{m-1} \ 4^2  \cdot g \left( ( \mathcal{W}_{ij}) \widetilde{\mathcal{I}}_k, \Xi_{\gamma} \right)^2 \\
 = & \ 48 \cdot 2 \sum_{k =1}^{m-1}  \left( \pm \frac{1}{ \sqrt{2} \sqrt{k^2+k}} \right)^2 \left\lbrace \sum_{k+1=i< j \leq m} k^2 \ + \sum_{1 \leq i < k+1 < j \leq m} 1 \ + \sum_{1 \leq i < k+1=j} (k+1)^2 \right\rbrace \\
 = & \ 48m(m-1).
\end{align*}

Overall we compute
\begin{align*}
| R_W^{\mathfrak{sp}(m) \oplus \mathfrak{sp}(1)} |^2 = & \  2 \ \sum_{\Xi_{\gamma} \in \mathcal{B}_0}  \ \ \sum_{\Xi_{\alpha} \in \mathcal{B}_1 \cup \mathcal{B}_4} \sum_{\Xi_{\beta} \in \mathcal{B}_2 \cup \mathcal{B}_3} \ \ ( \lambda_{\alpha} - \lambda_{\beta} )^2 \ g( ( \Xi_{\alpha} ) \Xi_{\beta}, \Xi_{\gamma} )^2 \\
= & \ 12m^2(m-1)+24m(m-1)+24m(m-1)^2+48m(m-1) \\
= &  \ 12m(m-1)(3m+4).
\end{align*}

\end{proof}

\begin{corollary}
Let $m \geq 2.$ An algebraic quaternion K\"ahler curvature tensor $R \in \operatorname{Sym}_B^2(\mathfrak{sp}(m) \oplus \mathfrak{sp}(1))$ satisfies
\begin{align*}
| R^{\mathfrak{sp}(m) \oplus \mathfrak{sp}(1)} |^2 = \frac{4}{3}(3m+4) | R_0 |^2.
\end{align*}
In particular, $R^{\mathfrak{sp}(m) \oplus \mathfrak{sp}(1)} = 0$ if and only of $R$ is a multiple of $R_{\mathbb{HP}^m}.$
\label{ReducedHatQKCurvatureTensors}
\end{corollary}
\begin{proof}
Recall that $L R_{\mathbb{HP}^m} = 0$ for all $L \in \mathfrak{sp}(m) \oplus \mathfrak{sp}(1)$ due to proposition \ref{HatIsotropyIrreducibleSpaces}. Therefore the decomposition of $\operatorname{Sym}_B^2(\mathfrak{sp}(m) \oplus \mathfrak{sp}(1))$ into orthogonal, $Sp(m) \cdot Sp(1)$-invariant, irreducible subspaces implies that there is a constant $c \in \R$ such that
\begin{align*}
| R^{\mathfrak{sp}(m) \oplus \mathfrak{sp}(1)} |^2 = c \cdot | R_0 |^2 = c \cdot \left( |R|^2 - \frac{\scal(R)^2}{\scal(R_{\mathbb{HP}^m})^2} |R_{\mathbb{HP}^m}|^2 \right).
\end{align*}
For the curvature operator $R_W$ of the Wolf space $\frac{SO(m+4)}{S(O(m) \times O(4))}$ of proposition \ref{ReducedHatWolfSpace} we find
\begin{align*}
12m(m-1)(3m+4) = c \cdot \left( 2m(7m-4) - m(5m+1) \right) = c \cdot 9 m (m-1)
\end{align*}
and the claim follows.
\end{proof}

Theorem \ref{MainQKTheorem} is an immediate consequence of

\begin{proposition}
Let $(M,g)$ be a compact quaternion K\"ahler manifold of real dimension $4m \geq 8.$ Let $\mu_1 \leq \ldots \leq \mu_{m(2m+1)+3}$ denote the eigenvalues of the corresponding quaternion K\"ahler curvature operator.
If 
\begin{align*}
\mu_{1} + \ldots + \mu_{\floor{\frac{m+1}{2}}} + \frac{5+(-1)^{m} \cdot 3}{12} \cdot \mu_{\floor{\frac{m+1}{2}}+1}  \geq 0,
\end{align*}
then $(M,g)$ is locally symmetric.
\label{BochnerTechniqueQKmaniolds}
\end{proposition}
\begin{proof}
Quaternion K\"ahler manifolds in real dimension $4m \geq 8$ are Einstein. Hence the curvature tensor $R$ is harmonic and thus satisfies the Bochner formula
\begin{align*}
\Delta \frac{1}{2} | R |^2 = | \nabla R |^2 + \frac{1}{2} \cdot g \left( \mathfrak{R} \left( R^{\mathfrak{sp}(m) \oplus \mathfrak{sp}(1)} \right), R^{\mathfrak{sp}(m) \oplus \mathfrak{sp}(1)} \right)
\end{align*}
due to corollary \ref{BochnerTechnique}.

For algebraic quaternion K\"ahler curvature operators $R \in \operatorname{Sym}_B^2(\mathfrak{sp}(m) \oplus \mathfrak{sp}(1))$, corollary \ref{ReducedHatQKCurvatureTensors} and \cite[Lemma 2.2]{PetersenWinkNewCurvatureConditionsBochner} imply
\begin{align*}
| L R |^2 = |L R_0 |^2 \leq 8 |L|^2 | R_0 |^2 = \frac{6}{3m+4} |L|^2 | R^{\mathfrak{sp}(m) \oplus \mathfrak{sp}(1)} |^2 
% = \left( \frac{m}{2} + \frac{2}{3} \right)^{-1} |L|^2 | R^{\mathfrak{sp}(m) \oplus \mathfrak{sp}(1)} |^2
\end{align*}
for every $L \in \mathfrak{sp}(m) \oplus \mathfrak{sp}(1)$. Note that $\floor{\frac{m}{2} + \frac{2}{3}} = \floor{\frac{m+1}{2}}$ and $\frac{m}{2} + \frac{2}{3} -  \floor{\frac{m}{2} + \frac{2}{3}}  = \frac{5+(-1)^{m} \cdot 3}{12}.$

Due to \cite[Lemma 1.8] {PetersenWinkVanishingHodgeNumbers}, the assumption on the eigenvalues of the quaternion K\"ahler curvature operator implies that
\begin{align*}
g \left( \mathfrak{R} \left( R^{\mathfrak{sp}(m) \oplus \mathfrak{sp}(1)} \right), R^{\mathfrak{sp}(m) \oplus \mathfrak{sp}(1)} \right) \geq 0.
\end{align*}
Hence the maximum principle shows that $R$ is parallel.
\end{proof}

%\begin{comment}

%\end{comment}

%\bibliography{BochnerTechniqueReferences}
%\bibliographystyle{amsalpha}

\end{document}